\documentclass[11pt]{article}
\usepackage{amssymb,amsthm,amsmath,latexsym,geometry,tikz,setspace,enumerate}
\newtheorem{theorem}{Theorem}

\newtheorem{corollary}[theorem]{Corollary}
\newtheorem{conjecture}[theorem]{Conjecture}
\newtheorem{claim}{Claim}

\begin{document}

\onehalfspace

\title{Induced Matchings in Graphs of Maximum Degree 4}

\author{Felix Joos\thanks{Institut f\"{u}r Optimierung und Operations Research, 
Universit\"{a}t Ulm, Ulm, Germany,
e-mail: \texttt{felix.joos@uni-ulm.de}}}

\date{}

\maketitle

\vspace{-0.5cm}


\begin{abstract}
For a graph $G$, let $\nu_s(G)$ be the induced matching number of $G$.
We prove the sharp bound $\nu_s(G)\geq \frac{n(G)}{9}$ for every graph $G$
of maximum degree at most $4$ and without isolated vertices that does not contain a certain blown up $5$-cycle as a component.
This result implies a consequence of the well known conjecture of Erd{\H{o}}s and Ne{\v{s}}et{\v{r}}il,
saying that the strong chromatic index $\chi_s'(G)$ of a graph $G$ is at most $\frac{5}{4}\Delta(G)^2$,
because $\nu_s(G)\geq \frac{m(G)}{\chi_s'(G)}$ and $n(G)\geq \frac{m(G)\Delta(G)}{2}$.
Furthermore, it is shown that there is polynomial-time algorithm that computes induced matchings of size at least $\frac{n(G)}{9}$.
\end{abstract}


\section{Introduction}

For a graph $G$, a set $M$ of edges is an \emph{induced matching} of $G$ if
no two edges in $M$ have a common endvertex and no edge of $G$ joins two edges in $M$.
The maximum number of edges that form an induced matching in $G$ is the {\it strong matching number $\nu_s(G)$ of $G$}. 

Unlike the well investigated matching number \cite{lopl}, which can be determined in polynomial time \cite{ed},
it is known that the computation of the strong matching number is NP-hard even in very restricted graph classes as for example bipartite subcubic graphs \cite{ca1,lo,stva}.

The chromatic index $\chi'(G)$ and the strong chromatic index $\chi_s'(G)$ are the least numbers $k$
such that the edge set of $G$ can be partitioned in $k$ matchings and $k$ strong matchings, respectively.
While Vizing's Theorem gives $\chi'(G)\in\{\Delta(G),\Delta(G)+1\}$ \cite{vi}
where $\Delta(G)$ is the maximum degree of $G$,
no comparable result holds for the strong chromatic index.
In fact, Erd{\H{o}}s and Ne{\v{s}}et{\v{r}}il \cite{fashgytu2} conjectured $\chi_s'(G)\leq \frac{5}{4}\Delta(G)^2$,
which would be best-possible for even maximum degree 
and the graph obtained from a $5$-cycle by replacing every vertex by an independent set of order $\frac{\Delta(G)}{2}$.
In the case $\Delta(G)=4$, we denote this graph by $C_5^2$.
A simple greedy algorithm only gives $\chi_s'(G)\leq 2\Delta^2-2\Delta+1$,
and the best general result is due to Molloy and Reed who proved $\chi_s'(G)\leq 1.998\Delta(G)^2$ for sufficiently large maximum degree \cite{more}.

For subcubic graphs, Erd{\H{o}}s and Ne{\v{s}}et{\v{r}}il's conjecture was verified, to be precise $\chi_s'(G)\leq 10$ \cite{an, hoqitr}.
For $\Delta(G)=4$, Erd{\H{o}}s and Ne{\v{s}}et{\v{r}}il's conjecture claims $\chi_s'(G)\leq 20$ while the best known upper bound is $22$ \cite{cr}.
If $\Delta(G)\leq 4$, then their conjecture implies $\nu_s(G)\geq \frac{m(G)}{20}$.
In the present paper, I prove this consequence by showing $\nu_s(G)\geq \frac{n(G)}{10}$ if $\Delta(G)\leq 4$ and $G$ has no isolated vertices; note that $\frac{m(G)}{20}\leq \frac{\Delta(G)n(G)}{40}\leq \frac{n(G)}{10}$.
Furthermore, if $G$ does not contain $C_5^2$ as a component, then the result can be strengthened to $\nu_s(G)\geq \frac{n(G)}{9}$.
Both results are best possible.
Moreover, since the proof is constructive, it is easy to extract a polynomial-time algorithm which computes induced matchings of the guaranteed size.

For subcubic planar graphs, Kang, Mnich and M\"uller \cite{kamnmu} showed that $\nu_s(G)\geq \frac{m(G)}{9}$.
This was improved by Rautenbach et al.~\cite{jorasa} who proved $\nu_s(G)\geq \frac{m(G)}{9}$ for a subcubic graph $G$ without $K_{3,3}^+$ as a component
where $K_{3,3}^+$ is obtained from a $5$-cycle by replacing the vertices by independent sets of orders $1,1,1,2,$ and $2$, respectively; 
equivalently, $K_{3,3}^+$ is obtained from a $K_{3,3}$ by subdividing exactly one edge once.
In particular, $\nu_s(G)\geq \frac{m(G)}{9}=\frac{n(G)}{6}$ for a cubic graph $G$.
Recently, I proved  $\nu_s(G)\geq \frac{n(G)}{(\lceil\frac{\Delta}{2}\rceil+1) (\lfloor\frac{\Delta}{2}\rfloor+1)}$
if $G$ is a graph of sufficiently large maximum degree $\Delta$ and without isolated vertices \cite{jo}.

Our main result is the following.

\begin{theorem}\label{result four detail}
If $G$ is a graph of maximum degree at most $4$,
then
\begin{align*}\label{result4}
	\nu_s(G)\geq \frac{n(G)-i(G)-n_5(G)}{9}
\end{align*}
where $n_5(G)$ is the number of components of $G$ that are isomorphic to $C_5^2$
and $i(G)$ is the number of isolated vertices of $G$.
\end{theorem}

\noindent
Let $G$ be a graph with $\Delta(G)\leq 4$.
Since the graph $C_5^2$ has order $10$,
we obtain $n_5(G)\leq \frac{n(G)-i(G)}{10}$.
Moreover,
$\Delta(G)\leq 4$ implies $n(G)-i(G)\geq \frac{m(G)}{2}$.
Therefore, $\nu_s(G) \geq\frac{m(G)}{20}$ and if $n_5(G)=0$, then $\nu_s(G) \geq\frac{n(G)-i(G)}{9}$ and hence $\nu_s(G) \geq \frac{m(G)}{18}$.

In view of the graph $C_5^2$ and the graph obtained from a triangle by attaching two pendent vertices at every vertex, respectively,
Theorem \ref{result four detail} is best-possible.
However, I was not able to construct a graph $G$ without a component isomorphic to $C_5^2$ such that $\nu_s(G)= \frac{m(G)}{18}$.
Let $H$ be the graph obtained from a $5$-cycle by replacing the vertices by independent sets of orders $1,1,1,3,$ and $3$, respectively.
If $G$ is the graph obtained from two disjoint copies of $H$ by identifying the unique vertices of degree~$2$ in the two copies,
then $\nu_s(G)=\frac{34}{17}=\frac{m(G)}{17}$.

\begin{conjecture}
If $G$ is a graph of maximum degree at most $4$ and no component is isomorphic to $C_5^2$,
then $\nu_s(G)\geq \frac{m(G)}{17}$.
\end{conjecture}

\noindent
We use standard notation and terminology.
For a graph $G$, let $V(G)$ and $E(G)$ be its vertex set and edge set, respectively.
Let the \textit{order} and the \textit{size} of $G$ be defined by $|V(G)|$ and $|E(G)|$, respectively.
For a vertex $v$ of $G$, let $d_G(v)$ be its degree, 
$N_G(v)$ be the set of neighbors of $v$, and
$N_G[v]=N_G(v)\cup \{v\}$. 
If the graph is
clear from the context, we only write $d(v)$, $N(v)$, and $N[v]$, respectively.
If $d_G(v)=k$ holds for a non-negative integer $k$, then we say that $v$ is a \textit{degree}-$k$ vertex in $G$.
A set $I$ of vertices of $G$ is \textit{independent} if  no edge of $G$ joins two vertices in $I$.
Two edges $e$ and $f$ are independent if they do not share a common vertex and there is no edge that is adjacent to $e$ and $f$.
The rest of the paper is devoted to the proof of Theorem \ref{result four detail}.

\section{Proof of Theorem \ref{result four detail}}

For a contradiction, we assume that $G$ is a counterexample of minimum order.
Since the statement of the theorem is linear in terms of the components, $G$ is connected.
It is easy to see that $n_5(G)=i(G)=0$.
By a sequence of claims, we establish several properties of $G$ in order to derive a final contradiction.
All claims follow a common pattern.
We mark particular (pairwise independent) edges and delete all vertices $S$ of $G$ at distance at most 1 from these edges.
We denote the resulting graph by $G'$.
Note that $n_5(G')=0$.
By the choice of $G$, we know that $\nu_s(G')\geq \frac{1}{9}(n(G')-i(G'))$.
Afterwards, we obtain a contradiction by considering a maximum induced matching of $G'$ together with the marked edges of $G$;
we only have to show that $|S|+i(G')\leq 9k$ where $k$ is the number of marked edges.
In all our cases $k$ is $1$ or $2$.
Throughout the proof we denote by $I'$ the set of isolated vertices of $G'$.
Note that a vertex in $I'$ has all its neighbors in $S$.

\begin{claim}\label{c1}
If $v$ is a vertex of degree at least $2$,
then $v$ is adjacent to at least two vertices of degree at least $2$.
\end{claim}
\begin{proof}[Proof of Claim \ref{c1}]
For a contradiction, we assume that $v$ is adjacent to at most one vertex $w$ of degree at least $2$.
If $w$ does not exist, then $G$ is a path of order $3$, which is a contradiction.
Thus we may assume that $w$ exists.
Let $u$ be a degree-$1$ vertex adjacent to $v$ and we mark the edge $uv$.
Recall that $S$ is the set of vertices of $G$ that are at distance at most $1$ to some marked edge.
Thus $|S|\leq 5$.
Moreover,
all isolated vertices of $G'=G-S$ are adjacent to $w$.
This implies that $|S|+i(G')\leq 8$, which is a contradiction and completes the proof of Claim \ref{c1}.
\end{proof}

\begin{claim}\label{c2}
If $u_1$ and $u_2$ are distinct degree-$1$ vertices,
then $u_1$ and $u_2$ do not have a common neighbor.
\end{claim}
\begin{proof}[Proof of Claim \ref{c2}]
Assume for a contradiction that $v$ is the common neighbor of $u_1$ and $u_2$.
By Claim \ref{c1},
$v$ has degree $4$.
Let $w_1$ and $w_2$ be the neighbors of $v$ beside $u_1$ and $u_2$.
We mark the edge $u_1v$.
This implies that $|S|=5$.
By Claim \ref{c1}, $w_1$ and $w_2$ are adjacent to at most two degree-$1$ vertices;
that is, if $i(G')\geq 5$, 
then $i(G')=5$, $w_1$ and $w_2$ are adjacent to two degree-$1$ vertices, respectively, and there is a degree-$2$ vertex adjacent to both $w_1$ and $w_2$; 
thus $G$ is a graph of order $10$ and $\nu_s(G)=2$, which is a contradiction.
Therefore, $i(G')\leq 4$ and hence $|S|+i(G')\leq 9$, which is a contradiction.
\end{proof}

\begin{claim}\label{c3}
If $u_1$ and $u_2$ are degree-$1$ vertices,
then ${\rm dist}(u_1,u_2)\not=4$.
\end{claim}
\begin{proof}[Proof of Claim \ref{c3}]
Assume for a contradiction that ${\rm dist}(u_1,u_2)=4$ and let $v_i$ be the neighbor of $u_i$ for $i\in\{1,2\}$.
We mark $u_1v_1$ and $u_2v_2$ and hence $|S|\leq 9$.
Note that, by Claim \ref{c2}, there are at most five degree-$1$ vertices in $V(G)\setminus S$ adjacent to a vertex in $S$. 
Furthermore, there are at most $14$ edges joining $S$ and vertices of $G'$.
Thus $i(G')\leq 9$ and hence $|S|+i(G')\leq 18$.
\end{proof}

\begin{claim}\label{c4}
$\delta(G)\geq 2$.
\end{claim}
\begin{proof}[Proof of Claim \ref{c4}]
Assume for a contradiction that there is a degree-$1$ vertex $u$ in $G$ and let $v$ be its neighbor.
We mark $uv$.
If $d_G(v)\leq 3$, then Claim \ref{c2} immediately  implies $|S|+i(G')\leq 8$.
Thus we may assume that $w_1,w_2,w_3$ are the neighbors of $v$ beside $u$ and hence $|S|=5$.

Suppose $\{w_1,w_2,w_3\}$ is an independent set in $G$.
By Claim \ref{c2} and \ref{c3}, there is at most one degree-$1$ vertex in $V(G)\setminus S$ adjacent to a vertex in $S$.
Since there are at most nine edges joining $S$ and vertices of $G'$,
we have $i(G')\geq 5$ only if $G$ is a graph of order $10$ and exactly one degree-$1$ vertex and four degree-$2$ vertices are adjacent to $w_1,w_2,w_3$;
this implies that $\nu_s(G)\geq 2\geq \frac{n(G)}{9}$, which is a contradiction.

Suppose now that $G[\{w_1,w_2,w_3\}]$ is not a triangle.
By Claim \ref{c2} and \ref{c3}, there are at most two degree-$1$ vertices in $V(G)\setminus S$ adjacent to a vertex in $S$.
Since there are at most seven edges joining $S$ and vertices of $G'$,
we conclude $i(G')\leq 4$.

Suppose now that $G[\{w_1,w_2,w_3\}]$ is a triangle.
By Claim \ref{c3}, there are at most three degree-$1$ vertices in $V(G)\setminus S$ adjacent to a vertex in $S$.
Since there are at most three edges joining $S$ and vertices of $G'$,
we conclude $i(G')\leq 3$.
\end{proof}


\begin{claim}\label{c5}
Degree-$2$ vertices are adjacent only to degree-$4$ vertices.
\end{claim}
\begin{proof}[Proof of Claim \ref{c5}]
We assume for a contradiction that $u$ is a degree-$2$ vertex and $v$ is a neighbor of $u$ such that $d_G(v)\leq 3$.
We mark $uv$ and hence $|S|\leq 5$.
This implies that at most nine edges join $S$ and vertices of $G'$.
By Claim \ref{c4}, this implies $i(G')\leq 4$.
\end{proof}

\begin{claim}\label{c6}
Every degree-$4$ vertex is adjacent to at most two degree-$2$ vertices.
\end{claim}
\begin{proof}[Proof of Claim \ref{c6}]
Assume for a contradiction that there is a degree-$4$ vertex $v$ which has at least three neighbors of degree $2$.
Let $u$ be one of these neighbors and mark $uv$.
Thus $|S|\leq 6$.
Note that at most eight edges join $S$ and $V(G)\setminus S$.
Since $i(G')\geq 4$ implies that all isolated vertices of $G'$ have degree $2$ in $G$.
Thus a degree-$2$ vertex of $S$ and a degree-$2$ vertex of $I'$ share a common edge and this contradicts Claim \ref{c5}.
Thus we may assume that $i(G')\leq 3$ and so $|S|+i(G')\leq 9$, which is a contradiction.
\end{proof}

\begin{claim}\label{c7}
Every degree-$4$ vertex is adjacent to at most one degree-$2$ vertex.
\end{claim}
\begin{proof}[Proof of Claim \ref{c7}]
Assume for a contradiction that there is a degree-$4$ vertex $v$ with two neighbors $u_1,u_2$ of degree $2$.
Let $w$ be the neighbor of $u_1$ beside $v$.
We mark $u_1v$ and hence $|S|\leq 6$.
If $|S|\leq 5$, then there are at most six edges joining $S$ and $V(G)\setminus S$ and hence $i(G')\leq 3$.
Thus we may assume $|S|=6$.
For a contradiction, we assume that $i(G')\geq 4$.
Note that at most $10$ edges join $S$ and $I'$.
Suppose $u_2$ is adjacent to a vertex in $I'$, then, by Claim \ref{c5}, $I'$ contains a vertex of degree $4$.
Thus $I'$ contains three degree-$2$ vertices.
However, $w$ is adjacent to two of them and hence in total adjacent to at least three degree-$2$ vertices, which is a contradiction to Claim \ref{c6}.
Thus we may assume that $u_2$ is not adjacent to a vertex in $I'$.
Note that $I$ either contains four degree-$2$ vertices or three degree-$2$ vertices and one degree-$3$ vertex.
In both cases $w$ is adjacent to at least two degree-$2$ vertices in $I'$, which is a contradiction to Claim \ref{c6}.
\end{proof}

\begin{claim}\label{c8}
$\delta(G)\geq 3$.
\end{claim}
\begin{proof}[Proof of Claim \ref{c8}]
For a contradiction, we assume that there is a degree-$2$ vertex $u$ and if possible choose $u$ to be contained in a $C_4$.
Let $v,w$ be the neighbors of $u$.
We mark $uv$ and hence $|S|\leq 6$.
If $|S|\leq 5$, then at most eight edges join $S$ and $I'$, which implies that $i(G')\leq 4$, which is a contradiction.
Thus we may assume $|S|=6$.
If the graph $G[S]$ has size at least $7$,
then there are at most eight edges joining $S$ and $I'$ and because $w$ is only adjacent to vertices of degree at least $3$ (Claim \ref{c7}),
we obtain $i(G')\leq 3$, which is a contradiction.

Suppose now that $G[S]$ is a graph of size $6$.
Hence at most $10$ edges join $S$ and $I'$.
Assume for contradiction that $i(G')\geq 4$. 
If $i(G')\geq 5$, then Claim \ref{c7} yields the contradiction.
Thus we assume that $i(G')=4$ and hence $I'$ contains at least two degree-$2$ vertices $x,y$.
By Claim \ref{c7}, $x,y$ have distinct neighbors in $S$ and hence $w$ is adjacent to $x$ or $y$.
However, $w$ is adjacent to $u$, which is a contradiction to Claim \ref{c7}.

Thus we may assume that $G[S]$ is a graph of size $5$; that is, $G[S]$ is a tree and thus $u$ is not contained in a $C_4$.
Moreover, by our choice of $u$, no degree-$2$ vertex is contained in a $C_4$.
This implies that $I'$ contains no degree-$2$ vertex 
because such a vertex cannot be adjacent to $w$ (Claim \ref{c7}) and if both neighbors in $S$ are distinct from $w$, then it is contained in a $C_4$.
Note that at most $12$ edges join $S$ and $I'$.
If $i(G')\geq 4$, then $i(G')=4$ and all vertices in $I'$ are degree-$3$ vertices and all vertices in $S\setminus\{u,v\}$ are degree-$4$ vertices.
Thus $n(G)=10$ and $\nu_s(G)\geq 2$, which is a contradiction.
\end{proof}

\begin{claim}\label{c9}
The set of degree-$3$ vertices is an independent.
\end{claim}
\begin{proof}[Proof of Claim \ref{c9}]
Assume for a contradiction that two degree-$3$ vertices $u,v$ are adjacent.
We mark $uv$.
Note that $|S|\leq 6$ and
at most $12$ edges join $S$ and $I'$.
If $|S|+i(G')\geq 10$, then $|S|=6$, $i(G')=4$, all vertices in $I'$ and $S\setminus\{u,v\}$ are degree-$3$ and degree-$4$ vertices, respectively, and $n(G)=10$.
It is easy to see that $\nu_s(G)\geq 2$, which is a contradiction.
\end{proof}

\begin{claim}\label{c10}
No degree-$3$ vertex is contained in a triangle.
\end{claim}
\begin{proof}[Proof of Claim \ref{c10}]
Assume for a contradiction that a degree-$3$ vertex $u$ is contained in a triangle $uvwu$.
We mark $uv$.
Note that $|S|\leq 6$ and at most $11$ edges join $S$ and $I'$.
This implies that $i(G')\leq 3$.
\end{proof}

\begin{claim}\label{c11}
$G$ is not a graph of order $10$ and minimum degree $3$.
\end{claim}
\begin{proof}[Proof of Claim \ref{c11}]
We show that $\nu_s(G)\geq 2$ holds for every connected graph $G\not=C_5^2$ of order $10$ with $\delta(G)=3$ such that the set of degree-$3$ vertices form an independent set and every degree-$3$ vertex is not contained in a triangle.

Since the number of degree-$3$ vertices is even, we suppose first that there are two degree-$3$ vertices $u_1,u_2$.
If ${\rm dist}(u_1,u_2)\geq 4$, then $\nu_s(G)\geq 2$ trivially holds.

Note that for every edge $xy$, we may assume that the graph $G-(N[x]\cup N[y])$ is an independent set.
Let $v_1, v_2, v_3$ be the neighbors of $u_1$.

Suppose ${\rm dist}(u_1,u_2)= 3$.
Let $w_1,w_2,w_3$ be the neighbors of $v_1$ beside $u_1$.
We mark $u_1v_1$ and thus $S=N[u_1]\cup N[v_1]$.
Since ${\rm dist}(u_1,u_2)= 3$, we conclude $u_2\notin S$ and hence $N(u_2)=\{w_1,w_2,w_3\}$.
In order that $\{u_2w_i , u_1v_j\}$ for $i\in \{1,2,3\}$ and $j\in \{1,2\}$ is not an induced matching of size two,
we conclude that both $v_2$ and $v_3$  are adjacent to $w_1,w_2,w_3$.
Thus at most six edges leave $S$ but exactly $10$ edges leave $I$ towards $S$, which is a contradiction.

Suppose ${\rm dist}(u_1,u_2)= 2$.
By symmetry, let $v_1$ be a common neighbor of $u_1$ and $u_2$.
We mark $u_1v_1$ and thus $S=N[u_1]\cup N[v_1]$.
Note that $V(G)\setminus S$ is a set of three degree-$4$ vertices in $G$.
Hence, by using that there are $12$ edges leaving $V(G)\setminus S$, there is exactly one edge within $S\setminus \{u_1,v_1\}$.
By symmetry, we assume that $v_2$ has only neighbors in $V(G)\setminus S$; that is, $v_2$ is adjacent to all vertices in $V(G)\setminus S$.
Moreover, $u_2$ has at least one non-neighbor $w$ in $V(G)\setminus S$.
This implies that $\{u_2v_1,v_2w\}$ is an induced matching of size $2$.
This completes the case that $G$ contains exactly two degree-$3$ vertices.

Next, we suppose that $G$ contains exactly four degree-$3$ vertices $u_1,\ldots,u_4$.

Suppose there is a vertex $v_1$ that is adjacent to $u_1,\ldots,u_4$.
We mark $u_1v_1$.
Since $12$ edges leave $V(G)\setminus S$, the graph $G[S]$ is a tree.
Let $w$ be a non-neighbor of $u_2$ in $V(G)\setminus S$ and $v_2$ be a neighbor of $u_1$ beside $v_1$.
Since $w$ has four neighbors $wv_2\in E(G)$.
This implies that $\{u_2v_1,v_2w\}$ is an induced matching of size $2$.

Suppose there is a vertex $v_1$ that is adjacent to exactly three degree-$3$ vertices, say $u_1,u_2,u_3$.
Let $v_2,v_3$ be the neighbors of $u_1$ beside $v_1$.
We mark $u_1v_1$. 
Hence $u_4$ is contained in the independent set $V(G)\setminus S$ and adjacent to $v_2,v_3$ and the degree-$4$ neighbor of $v_1$.
Hence, by using that there are $11$ edges leaving $V(G)\setminus S$, there is exactly one edge within $S\setminus \{u_1,v_1\}$.
By symmetry, we assume that $u_2v_2\notin E(G)$.
This implies that $\{u_2v_1,u_4v_2\}$ is an induced matching of size $2$.

Suppose there is a vertex $v_1$ that is adjacent to exactly two degree-$3$ vertices, say $u_1,u_2$, but no vertex is adjacent to more than two degree-$3$ vertices.
This implies that $u_3,u_4\in V(G)\setminus S$.
Let $v_2,v_3$ be the neighbors of $u_1$ beside $v_1$.
We mark $u_1v_1$. 
Since no degree-$3$ vertex is contained in a triangle and $\{u_1,\ldots,u_4\}$ is an independent set,
$u_2$ is adjacent to $v_2$ or $v_3$; by symmetry, say $v_2$.
Because there are $10$ edges leaving $V(G)\setminus S$,
there are exactly two edges within $S\setminus \{u_1,v_1\}$.
Thus there is a degree-$4$ neighbor $x$ of $v_1$ that has only neighbors in $V(G)\setminus S$.
Furthermore, there is a non-neighbor $w$ of $v_2$ in $V(G)\setminus S$ (note that $w=u_3$ or $w=u_4$ is possible).
This implies that $\{u_1v_2,wx\}$ is an induced matching of size $2$.

Suppose now that all degree-$4$ vertices are adjacent to at most one degree-$3$ vertex.
This implies that there are at least $3$-times as many degree-$4$ vertices as degree-$3$ vertices, which is a contradiction that $G$ has order $10$.
This completes the case that $G$ contains exactly four degree-$3$ vertices.

If $G$ contains at least six degree-$3$ vertices, then at least $18$ edges leave the set of degree-$3$ vertices,
but at most $16$ edges leave the set of degree-$4$ vertices, which is the final contradiction.
\end{proof}

A vertex $v$ is a \textit{cut vertex} of a graph $G$ if $G-v$ has more components than $G$
and a \textit{block} is a maximal $2$-connected subgraph of $G$.

\begin{claim}\label{c12}
There is no cut vertex $v$ such that there is a block $B$ of order $10$ with $v\in V(B)$ such that $B$ contains only one cut vertex of $G$.
\end{claim}
\begin{proof}[Proof of Claim \ref{c12}]
For a contradiction, we assume that such a configuration exists.
Suppose there is an edge $f$ in $B$ at distance at least $2$ from $v$.
We mark $f$ and thus $|S|+i(G')\leq 9$, which is a contradiction.
Thus all edges in $B$ are at distance at most $1$ from $v$.
Note that
there are at most three vertices at distance $1$ from $v$ and hence at most three vertices at distance $2$ from $v$, which is a contradiction.
\end{proof}

Note the following: suppose we mark an edge incident to a degree-$3$ and degree-$4$ vertex.
Thus $|S|=7$.
If $i(G')=3$, then Claim \ref{c11} and \ref{c12} imply that $G'$ has non-trivial components and at least two edges join $S$ and these non-trivial components.
This simplifies the proofs of the next claims significantly.

\begin{claim}\label{c13}
No degree-$4$ vertex is adjacent to four degree-$3$ vertices.
\end{claim}
\begin{proof}[Proof of Claim \ref{c13}]
For a contradiction, we assume that there is a vertex $v$ that is adjacent to four degree-$3$ vertices $u_1,\ldots, u_4$.
We mark $u_1v$; that is, $|S|=7$ and at most $12$ edges join $S$ and $I'$.
Every vertex in $I'$ has degree $4$ because there are only two degree-$4$ vertices which might be adjacent to a vertex in $I'$ (Claim \ref{c9}).
Thus $i(G')=3$ and $G$ is a graph of order $10$, which is a contradiction to Claim \ref{c11}.
\end{proof}

\begin{claim}\label{c14}
No degree-$4$ vertex is adjacent to three degree-$3$ vertices.
\end{claim}
\begin{proof}[Proof of Claim \ref{c14}]
For a contradiction, we assume that there is a vertex $v$ that is adjacent to three degree-$3$ vertices $u_1,u_2,u_3$.
If possible choose $v$ and $u_1$ such that $u_1v$ is contained in a $C_4$.
Let $v_1,v_2$ be the neighbors of $u_1$ beside $v$.
We mark $u_1v$; that is, $|S|=7$ and at most $13$ edges join $S$ and $I'$.
If $i(G')\geq 4$, then two degree-$3$ vertices share a common edge which contradicts Claim \ref{c9}.
For contradiction, we assume that $i(G')=3$.
By Claim \ref{c11} and \ref{c12}, at least two edges join $S$ and $V(G)\setminus (S \cup I')$.
Hence at most $11$ edges join $S$ and $I'$.
Thus there is at least one degree-$3$ vertex $w\in I'$ that is adjacent to the three degree-$4$ vertices in $S$.
If there are three degree-$3$ vertices in $I'$, then all are adjacent to $v_1$ which is a contradiction to Claim \ref{c13}.
If there are two degree-$3$ vertices in $I'$, then $S$ induces a tree and both degree-$3$ vertices in $I'$ are adjacent to both $v_1$ and $v_2$.
Thus $v_1$ is a vertex adjacent to three degree-$3$ vertices and $v_1w$ is contained in a $C_4$, which is a contradiction to our choice of $u_1$ and $v$.

Therefore,
$I'$ contains two degree-$4$ vertices and the degree-$3$ vertex $w$, and exactly two edges join $S$ and $V(G)\setminus(S \cup I')$.
Thus there are two vertices $x,y\notin \{u_1,v\}$ such that one is a neighbor of $u_1$, one is a neighbor of $v$, and
$x$ has no neighbor in $V(G)\setminus(S \cup I')$ but $y$ does.   
Marking $x$ with its neighbor in $\{u_1,v\}$ instead of $u_1v$ leads to a contradiction and this completes the proof of the claim.
\end{proof}

\begin{claim}\label{c15}
No degree-$4$ vertex is adjacent to two degree-$3$ vertices.
\end{claim}
\begin{proof}[Proof of Claim \ref{c15}]
For a contradiction, we assume that there is a vertex $v$ that is adjacent to two degree-$3$ vertices $u_1,u_2$.
If possible choose $v$ and $u_1$ such that $u_1v$ is contained in a $4$-cycle $C$ and if possible choose $C$ to contain two degree-$3$ vertices.
Let $v_1,v_2$ be the neighbors of $u_1$ beside $v$.
We mark $u_1v$; that is, $|S|=7$ and at most $14$ edges join $S$ and $I'$.
For a contradiction, we assume $i(G')\geq 3$.
Let $x_1,x_2$ be the degree-$4$ neighbors of $v$.
Note that $v_1,v_2$ and $x_1,x_2$ are adjacent to at most one degree-$3$ vertex and to at most two degree-$3$ vertices in $I'$, respectively;
that is, $I'$ contains at most two degree-$3$ vertices.
If $I'$ contains two degree-$3$ vertices $w_1,w_2$,
then, by symmetry of $v_1, v_2$, we obtain $N(w_i)=\{v_i,x_1,x_2\}$.
Thus $x_1$ is a degree-$4$ vertex adjacent to two degree-$3$ vertices and $x_1w_1x_2w_2x_1$ is a $4$-cycle containing two degree-$3$ vertices.
Our choice of $v$ implies that there is an edge joining $u_2$ and $v_1$ or $v_2$, which is a contradiction to Claim \ref{c14}.

Thus we assume that $I'$ contains at most one degree-$3$ vertex.
A degree counting argument implies that $i(G')=3$ and hence at least $11$ edges join $S$ and $I'$.
Claim \ref{c11} and \ref{c12} imply that there are at most $12$ edges joining $S$ and $I'$.
It follows that $S$ induces a tree.

Suppose $v_1$ has no neighbor in $V(G)\setminus (S\cup I')$.
Thus $N(v_1)=I'\cup \{u_1\}$.
Let $w$ be a non-neighbor of $u_2$ in $I'$.
Marking $v_1w$ instead of $u_1v$ leads to a contradiction because the edges $v_1w$ and $u_1v$ are independent
and at most one vertex in $G- (S\cup I')$ becomes an isolated vertex.

Thus, by symmetry in $\{v_1,v_2\}$, both $v_1$ and $v_2$ have a neighbor in $V(G)\setminus (S\cup I')$.
By symmetry in $\{x_1,x_2\}$, we may assume that $x_1$ has no neighbor in $V(G)\setminus (S\cup I')$.
Hence $N(x_1)=I'\cup \{v\}$.
Let $w'$ be a non-neighbor of $v_1$ in $I'$.
Similar as above, marking $x_1w'$ instead of $u_1v$ leads to a contradiction.
\end{proof}

\begin{claim}\label{c16}
$\delta(G)\geq 4$.
\end{claim}
\begin{proof}[Proof of Claim \ref{c16}]
For a contradiction, we assume that there is a vertex $v$ which is adjacent to a degree-$3$ vertices $u$.
Choose $u,v$ such that $uv$ is contained in as many $4$-cycles as possible.
Let $v_1,v_2$ be the neighbors of $u$ beside $v$ and $x_1,x_2,x_3$ be the neighbors of $v$ beside $u$.
We mark $uv$; that is, $|S|=7$ and at most $15$ edges join $S$ and $I'$.

If $I'$ does not contain a degree-$3$ vertex, then $i(G')\leq 3$.
Suppose $I'$ contains a degree-$3$ vertex $w$. 
By Claim \ref{c15}, we conclude $N(w)=\{x_1,x_2,x_3\}$
and thus, by our choice of $uv$,
the set $S\setminus\{u,v\}$ induces a graph of size at least $2$,
which in turn implies that at most $11$ edges join $S$ and $I'$.
Hence $i(G')\leq 3$.



Therefore,
we may assume that $i(G')=3$ and Claim \ref{c11} and \ref{c12} imply that there are at most $13$ edges joining $S$ and $I'$.
If $I'$ contains a degree-$3$ vertex, then with the same argumentation as above there are at least two edges within $S\setminus \{u,v\}$ but then at most nine edges join $S$ and $I'$,
which is a contradiction to the fact there is at most one degree-$3$ vertex in $I'$.

Thus we may assume that $I'$ contains three degree-$4$ vertices.
A degree sum argument implies that $S$ induces a tree and exactly three edges join $S$ and $V(G)\setminus(S \cup I')$.
If this three edges are incident with a common vertex $z\notin S$ and $z$ has degree~$3$, then $z$ is contained in a $4$-cycle and this contradicts the choice of $u$ and $v$ because $S$ induces a tree.
Thus deleting vertices in $S$ does not lead to isolated vertices in $G-(S \cup I')$.

Suppose $v_1$ or $v_2$, by symmetry say $v_1$, has at least one neighbor in $V(G)\setminus(S \cup I')$.
Let $w\in I'$ be a non-neighbor of $v_1$.
By symmetry in $\{x_1,x_2,x_3\}$, we conclude that $x_1$ has no neighbor in $V(G)\setminus(S \cup I')$ and hence marking $x_1w$ instead of $uv$ leads to a contradiction.

Therefore, we may assume that $N(v_i)=\{u\}\cup I'$ for $i\in \{1,2\}$.
By symmetry, $x_1$ has a neighbor in $V(G)\setminus(S \cup I')$ and a non-neighbor $w\in I'$.
Marking $v_1w$ instead of $uv$ leads to a contradiction, which completes the proof of the claim.
\end{proof}

\begin{claim}\label{c17}
$G$ triangle-free.
\end{claim}
\begin{proof}[Proof of Claim \ref{c17}]
For a contradiction, we assume that there is an edge $uv$ which is contained in a triangle.
Choose $uv$ such that it is contained in as many triangles as possible.
We mark $uv$.
If $uv$ is contained in at least two triangles,
then $|S|\leq 6$ and at most $10$ edges join $S$ and $I'$.
Thus $i(G')\leq 2$, which is a contradiction.

Therefore, we assume that $uv$ is contained in one triangle $uvwu$ only.
Moreover, we choose the triangle edge $uv$ such that it is contained in as many $4$-cycles as possible.
Thus $|S|=7$ and at most $14$ edges join $S$ and $I'$.
Hence $i(G')=3$.
Furthermore, $S\setminus\{u,v\}$ induces a graph on at most one edge.
Let $x\in I'$.
If $xw \in E(G)$, then either $uw$ or $vw$ is contained in a triangle and in two $4$-cycles, which is a contradiction to our choice of $uv$.
This implies that $S\setminus \{u,v,w\} \cup I'$ induces a complete bipartite graph.
Let $y\in N(x)$. 
Marking $xy$ instead of $uv$ leads to a contradiction.
\end{proof}

Since we may assume from now on that $G$ is triangle-free,
we use the following notation in the remaining part of the proof.
The marked edge will be denoted by $uv$.
Moreover, let $N(u)=\{v,u_1,u_2,u_3\}$ and $N(v)=\{u,v_1,v_2,v_3\}$.
Note that all these vertices are distinct and that $|S|=8$.
Furthermore, let $S'=S\setminus\{u,v\}$.
This implies that at most $18$ edges join $I'$ and $S$.
For a contradiction, we assume that $i(G')\geq 2$; that is, at least eight edges join $I'$ and $S'$.
Let $I'=\{w_1,w_2,\ldots \}$.
Let $m_{S'}$ be the number of edges in $G[S']$.
If $m_{S'}\geq 6$, then at most six edges join $I'$ and $S'$ and this a contradiction.
Since $G$ is triangle-free, $G[S']$ is bipartite.

\begin{claim}\label{c18}
$\Delta(G[S'])\leq 2$.
\end{claim}
\begin{proof}[Proof of Claim \ref{c18}]
By symmetry, we assume for contradiction that $v_1$ is adjacent to $u_1,u_2,u_3$.
Suppose first that $m_{S'}\geq 4$ and hence $i(G')=2$.
By symmetry, $u_1v_2\in E(G)$.
If $u_1v_3 \in E(G)$, 
then $i(G')=2$ and $N(w_i)=\{u_2,u_3,v_2,v_3\}$ for $i\in\{1,2\}$.
Thus $G=C_5^2$, which is a contradiction.
Hence we suppose $u_1v_3\notin E(G)$.
We mark instead of $uv$ the edge $u_1v_1$.
If $u_1$ has a neighbor in $I'$, say $w_1$, then $w_2v_3$ is independent of $u_1v_1$, which leads to a contradiction.
If $u_1$ has no neighbor in $I'$, then $w_1v_3w_2$ is a path independent from $u_1v_1$, which leads to a contradiction.

Therefore, we may suppose $m_{S'}= 3$.
If $i(G')=3$, then $n(G)=11$ and $u_1$ has a non-neighbor in $I'$, say $w_1$.
Hence $uu_1$ together with $w_1v_2$ is an induced matching of size~$2$, which is a contradiction.
Thus $i(G')=2$ and four edges join $S'$ and $V(G)\setminus(S \cup I')$.
The fact that $i(G')=2$ and $m_{S'}= 3$ imply that $v_2$ and $v_3$ have at least one neighbor in $V(G)\setminus(S \cup I')$.
By symmetry in $\{u_1,u_2,u_3\}$, we assume that $u_1$ has no neighbor in $V(G)\setminus(S \cup I')$.
Thus marking $u_1v_1$ instead of $uv$ leads to a contradiction.
\end{proof}

\begin{claim}\label{c19}
$G[S']$ contains no two independent edges.
\end{claim}
\begin{proof}[Proof of Claim \ref{c19}]
By symmetry, we assume that $u_1v_1$ and $u_2v_2$ are independent edges.
This implies that at most $14$ edges join $S'$ and $I'$ and hence $i(G')\in \{2,3\}$.
We now mark $u_1v_1$ and $u_2v_2$ instead of $uv$.
Suppose first that $i(G')=3$.
Let $S''=N[u_1]\cup N[u_2]\cup N[v_1]\cup N[v_2]$ and $G''= G-S''$.
It is easily checked that $i(G'')+|S|\leq 18$, which is a contradiction.

Therefore, we may assume $i(G')=2$.
Suppose that $m_{S'}\geq 3$.
Hence at most four edges join $S$ and $V(G)\setminus(S \cup I')$.
This implies that $|S''\cup S \cup I'|\leq 14$ and at most $12$ edges join $S''\cup S \cup I'$ and $V(G)\setminus (S''\cup S \cup I')$;
that is, $|S''\cup S \cup I'|+i(G'')\leq |S''|+i(G'')\leq 17$, which is a contradiction.

Therefore, we may assume $m_{S'}=2$.
Since $G$ is triangle-free, we conclude that $w_i$ for $i\in \{1,2\}$ is adjacent to $u_3,v_3$ and to exactly one vertex of the two marked edges, respectively.
Thus both $u_3$ and $v_3$ have exactly one neighbor in $V(G)\setminus(S \cup I')$.
Moreover, exactly four edges join $\{u_1u_1, u_2v_2\}$ and $V(G)\setminus(S \cup I')$;
that is, $|S''\cup S \cup I'|\leq 14$ and at most $14$ edges join $S''\cup S \cup I'$ and $V(G)\setminus(S'' \cup S \cup I')$,
which implies the contradiction $|S''\cup S \cup I'|+i(G'')\leq |S''|+i(G'')\leq 17$.
\end{proof}

Claim \ref{c19} implies that $G[S']$ has at most one nontrivial component.
Claims \ref{c17}, \ref{c18} and \ref{c19} imply that the nontrivial component of $G[S']$, if it exists,  is
a $C_4$, a $P_4$, a $P_3$ or a $P_2$.

\begin{claim}\label{c20}
If it exists, then the nontrivial component of $G[S']$ is not a $C_4$.
\end{claim}
\begin{proof}[Proof of Claim \ref{c20}]
For a contradiction, we assume that the nontrivial component of $G[S']$ is a $C_4$.
By symmetry, $u_1v_1u_2v_2u_1$ is this $C_4$.
Since at most $10$ edges join $S$ and $I'$, we have $i(G')=2$.
Furthermore, since $G$ is triangle-free and by symmetry in $\{w_1,w_2\}$, we may assume that $N(w_1)=\{u_1,u_2,u_3,v_3\}$ and $N(w_2)=\{u_3,v_1,v_2,v_3\}$.
Marking $u_1v_1$ instead of $uv$ leads to a contradiction to Claim \ref{c19}.
\end{proof}

\begin{claim}\label{c21}
If it exists, then the nontrivial component of $G[S']$ is not a $P_4$.
\end{claim}
\begin{proof}[Proof of Claim \ref{c21}]
For a contradiction, we assume that the nontrivial component $P$ of $G[S']$ is a $P_4$.
By symmetry, $u_1v_2u_2v_1$ is this $P_4$; that is,
at most $12$ edges join $S$ and $I'$ and hence $i(G')\in \{2,3\}$.
Suppose first $i(G')=3$.
This implies that $n(G)=11$ and $N(u_3)=\{w_1,w_2,w_3,u\}$; by symmetry, say $w_1$ is nonadjacent to $u_2$ and $v_2$.
Since $\{u_3w_1,u_2v_2\}$ is an induced matching of size $2$, which is a contradiction, 
we may assume that $i(G')=2$;
that is, exactly four edges join $S$ and $V(G)\setminus(S \cup I')$.
Since $w_i$ for $i\in\{1,2\}$ has at most two neighbors in $P$,
we conclude that $w_i$ is adjacent to $u_3$ and $v_3$.
Moreover,
$u_3$ and $v_3$ have a neighbor in $V(G)\setminus(S \cup I')$.
Suppose at least one of the endvertices of $P$ is adjacent to both $w_1$ and $w_2$, say $u_1$.
Marking $uu_1$ instead of $uv$ leads to a contradiction because $v_3$ has a neighbor in $V(G)\setminus(S \cup I')$.
Thus we may assume that $u_2$ and $v_2$ have a neighbor in $I'$.
Then, marking $u_2v_2$ instead of $uv$ leads to a contradiction because $u_3$ and $v_3$ have a neighbor in $V(G)\setminus(S \cup I')$.
\end{proof}

In the following we choose $uv$ such that $m_{S'}$ is maximal.

\begin{claim}\label{c22}
If it exists, then the nontrivial component of $G[S']$ is not a $P_3$.
\end{claim}
\begin{proof}[Proof of Claim \ref{c22}]
For a contradiction, we assume that the nontrivial component $P$ of $G[S']$ is a $P_3$.
By symmetry, $u_1v_1u_2$ is this $P_3$.
If $v_1$ has a neighbor in $I'$, say $w_1$, then $N(w_1)=\{u_3,v_1,v_2,v_3\}$ because $G$ is triangle-free.
The neighborhood of $vv_1$ induces a graph of size at least $4$, which is a contradiction to the previous claims.
Thus we may assume that $v_1$ has a neighbor in $V(G)\setminus(S \cup I')$.

If $u_1$ or $u_2$, say $u_1$ has a neighbor in $I'$, say $w_1$,
then since $w_1$ is adjacent to $u_2$ or $u_3$,
neighborhood of $uu_1$ induces a graph on at least three edges, which is a contradiction to the previous claims.
Thus we may assume that both $u_1$ and $u_2$ have two neighbors in $V(G)\setminus(S \cup I')$.

Since $G$ is $4$-regular, $w_1$ is adjacent to at least one vertex in $P$, which is a contradiction to our assumptions.
\end{proof}

\begin{claim}\label{c23}
If it exists, then the nontrivial component of $G[S']$ is not a $P_2$.
\end{claim}
\begin{proof}[Proof of Claim \ref{c23}]
For a contradiction, we assume that the nontrivial component of $G[S']$ is an edge.
By symmetry, let $u_1v_1$ be this edge.
If $u_1$ or $v_1$, say $u_1$, is adjacent to a vertex in $I'$, say $w_1$,
then $w_1$ is nonadjacent to $u_2,u_3,v_1$ according to our choice of $uv$, which is a contradiction
to the $4$-regularity of $G$.
This implies that the neighborhood of every vertex in $I'$ is $\{u_2,u_3,v_2,v_3\}$.
Since $i(G')\geq 2$,
the neighborhood of the edge $u_2w_1$ induces a graph on at least two edges, which is a contradiction to our choice of $uv$.
\end{proof}

Claims \ref{c17}-\ref{c23} imply that $G$ has girth at least $5$.
This implies that every vertex in $V(G)\setminus S$ has at most two neighbors in $S$ and hence $i(G')=0$, which is the final contradiction.
\qed


\begin{thebibliography}{0}
\bibitem{an}
L.D. Andersen,
The strong chromatic index of a cubic graph is at most 10,
Discrete Math. 108 (1992) 231-252.

\bibitem{ca1}
K. Cameron,
Induced matchings,
Discrete Appl. Math. 24 (1989) 97-102.

\bibitem{cr}
D.W. Cranston,
Strong edge-coloring of graphs with maximum degree 4 using 22 colors,
Discrete Math. 306 (2006) 2772-2778.

\bibitem{ed}
J. Edmonds,
Paths, trees, and flowers,
Canad. J. Math. 17 (1965) 449-467.

\bibitem{fashgytu2}
R.J. Faudree, R.H. Schelp, A. Gy\'{a}rf\'{a}s, and  Zs. Tuza,
The strong chromatic index of graphs,
Ars Comb. 29B (1990) 205-211.


\bibitem{hoqitr}
P. Hor\'{a}k, H. Qing, and W.T. Trotter,
Induced Matchings in Cubic Graphs,
J. Graph Theory 17 (1993) 151-160.

\bibitem{jorasa}
F. Joos, D. Rautenbach, and T. Sasse, Induced Matchings in Subcubic Graphs,
SIAM J. Discrete Math. 28 (2014) 468-473.

\bibitem{jo}
F. Joos, Induced Matchings in Graphs of Bounded Maximum Degree, 2014, arXiv:1406.2440.

\bibitem{kamnmu}
R.J. Kang, M. Mnich, and T. M\"{u}ller,
Induced matchings in subcubic planar graphs,
SIAM J. Discrete Math. 26 (2012) 1383-1411.

\bibitem{lopl}
L. Lov\'{a}sz and M.D. Plummer,
Matching Theory,
vol. 29, Annals of Discrete Mathematics, North-Holland, Amsterdam, 1986.

\bibitem{lo}
V.V. Lozin,
On maximum induced matchings in bipartite graphs,
Inf. Process. Lett. 81 (2002) 7-11.

\bibitem{more}
M. Molloy and B. Reed, A bound on the strong chromatic index of a graph,
J. Combin. Theory Ser. B 69 (1997) 103-109.

\bibitem{stva}
L.J. Stockmeyer and V.V. Vazirani,
NP-completeness of some generalizations of the maximum matching problem,
Inf. Process. Lett. 15 (1982) 14-19.

\bibitem{vi}
V.G. Vizing,
On an estimate of the chromatic class of a p-graph,
Diskret. Analiz. 3 (1964) 25-30.


\end{thebibliography}
\end{document}